\documentclass[11pt]{amsart}
\usepackage{amsmath,amssymb,amscd,amsthm,amsxtra,array}
\usepackage{color}
\usepackage[
colorlinks=true,
linkcolor=blue,
 citecolor=blue,
  urlcolor=blue,
]{hyperref}

\newtheorem{theorem}{Theorem}[section]
\newtheorem{lemma}[theorem]{Lemma}

\theoremstyle{definition}

\newtheorem{definition}[theorem]{Definition}

\theoremstyle{remark}

\numberwithin{equation}{section}

\def\Xint#1{\mathchoice
  {\XXint\displaystyle\textstyle{#1}}%
  {\XXint\textstyle\scriptstyle{#1}}%
  {\XXint\scriptstyle\scriptscriptstyle{#1}}%
  {\XXint\scriptscriptstyle\scriptscriptstyle{#1}}%
  \!\int}
\def\XXint#1#2#3{{\setbox0=\hbox{$#1{#2#3}{\int}$}
    \vcenter{\hbox{$#2#3$}}\kern-.5\wd0}}

\def\avgint{\Xint-}

\begin{document}

\title[Reverse H\"older Inequality on spaces of homogeneous type] {Sharp Reverse H\"older property for $A_\infty$ weights on spaces of homogeneous type}

\author{Tuomas Hyt\"onen}
\address{Department of Mathematics and Statistics, University of Helsinki, P.O.B. 68, FI-00014 Helsinki, Finland}
\email{tuomas.hytonen@helsinki.fi}

\author{Carlos P\'erez}
\address{Departamento de An\'alisis Matem\'atico,
Facultad de Matem\'aticas, Universidad de Sevilla, 41080 Sevilla,
Spain} \email{carlosperez@us.es}

\author{Ezequiel Rela}
\address{Departamento de An\'alisis Matem\'atico,
Facultad de Matem\'aticas, Universidad de Sevilla, 41080 Sevilla,
Spain} \email{erela@us.es}
\thanks{The first author is supported by the European Union through the ERC Starting Grant “Analytic-probabilistic methods for borderline singular integrals”, and by the Academy of Finland, grants 130166 and 133264. The second author is supported by the Spanish Ministry of Science and Innovation grant MTM2009-08934, the second and third authors are also supported by the Junta de Andaluc\'ia, grant FQM-4745.}



\begin{abstract} 
In this article we present a new proof of  a sharp Reverse H\"older Inequality for $A_\infty$ weights. 
Then we derive two applications: a precise open property of   Muckenhoupt  classes and, as a consequence of this last result, we obtain a simple proof of a sharp weighted bound for the Hardy-Littlewood maximal function involving $A_\infty$ constants:
\begin{equation*}
  \|M\|_{L^p(w)} \leq c\, \left( \frac{1}{p-1}  [w]_{A_p}[\sigma]_{A_\infty}\right)^{1/p},
\end{equation*}
where $1<p<\infty$, $\sigma=w^{\frac{1}{1-p}}$ and $c$ is a dimensional constant. Our approach allows us to extend the result to the context  of spaces of homogeneous type and prove a \emph{weak} Reverse H\"older Inequality which is still sufficient to prove the open property for $A_p$ classes and the $L^p$ boundedness of the maximal function. In this latter case, the constant $c$ appearing in the norm inequality for the maximal function depends only on the doubling constant of the measure $\mu$ and the geometric constant $\kappa$ of the quasimetric.   

\end{abstract}

\maketitle

\section{Introduction and main results}

\subsection{Introduction}
In this article we present a new proof of a sharp Reverse H\"older Inequality (RHI) for $A_\infty$ weights. The sharpness of the result relies on the precise dependence of the exponent involved on the $A_\infty$  constant  of the weight. We improve on the result from \cite{HP}, where there is proved a sharp RHI for $\mathbb{R}^d$.  We present here a new and simpler approach, with the extra advantage of allowing us to extend the result to any space of homogeneous type. In  this case, we obtain a \emph{weak}  RHI which is still sharp in the dependence on the $A_\infty$ constant of the weight.
Furthermore this new approach gives a better result within the simplest context  which for the sake of clarity we present first in Section \ref{simplecase}  both the statement and proof.  The rest of the paper is devoted to the general case, namely when the underlying space is of homogeneous type.

We also present two results derived from the sharp RHI: first, a precise open property for $A_p$ weights (Theorem \ref{thm:Apselfimprove}), which is new even in the  standard case of $\mathbb{R}^d$ for $A_p$ classes defined over cubes. By this we mean that we exhibit, for the classical  $A_p\Rightarrow A_{p-\varepsilon}$ theorem, a quantitative analysis of the dependence of $\varepsilon$ on the $A_\infty$ constant of the weight. This sort of result was already stated in \cite{Buckley}, but here we present an improvement of that estimate. As a second consequence of our main result, we provide a simple proof of a  mixed weighted bound for the Hardy-Littlewood maximal function (Theorem \ref{thm:CorsharpBuckley}), originally proved in \cite{HP} in the usual context of  $\mathbb{R}^d$.    There are two reasons why this new result is of interest. First because it gives a new refinement of the well known improvement of  Muckenhoupt's classical theorem due to S. Buckley \cite{Buckley}. The second reason is because it also gives an interesting an unexpected improvement of the so called $A_2$ theorem (\cite{Hytonen:A2} \cite{HPTV})  as shown in \cite{HP} (see also \cite{HL-Ap-Ainf} and \cite{HLP}).

One of the main difficulties arising in the setting of spaces of homogeneous spaces is the absence of dyadic cubes, which is a useful and commonly used tool in analysis on metric spaces. In \cite{Christ}, Christ developed a substitute for dyadic cubes which has been exploited since then to overcome this obstacle. See, for instance, \cite{ABI-comparison},  where the authors study the relation between dyadic and classical maximal functions. There is also proved a qualitative reverse H\"older  inequality, which yields the standard $A_p \Rightarrow A_{p-\varepsilon}$ theorem. See  the recent work \cite{HytKai} for a construction of dyadic systems in this generality with application to weighted inequalities.

However, we present here an approach avoiding the use of dyadic sets, and we work directly with the natural quasimetric, following carefully the dependence on the geometric constants.  Some of the main ideas come from \cite{MacManus-Perez-98} where it was crucial to avoid the dyadic sets in order to get sharp bounds. 

Let us start with some standard definitions. A quasimetric $d$ on a set $\mathcal{S}$ is a function $d:{\mathcal S} \times
{\mathcal S} \rightarrow [0,\infty)$ which satisfies
\begin{enumerate}
 \item $d(x,y)=0$ if and only if $x=y$;
\item $d(x,y)=d(y,x)$ for all $x,y$;
 \item there exists a finite constant $\kappa \ge 1$ such that, for all $x,y,z \in \mathcal{S}$,
\begin{equation*}
d(x,y)\le \kappa (d(x,z)+d(z,y)).
\end{equation*}

\end{enumerate}
As usual, given $x \in \mathcal{S}$ and $r > 0$,  let $B(x,r) = \{y \in {\mathcal{S}} :d(x,y)
< r\}$ be the ball with center $x$ and radius $r$.  If $B=B(x,r)$ is a ball, we denote its radius $r$ by $r(B)$ and its center $x$ by $x_B$. A space of homogeneous type $({\mathcal{S}},d,\mu)$ is a set $\mathcal{S}$ together
with a quasimetric $d$ and a nonnegative Borel measure $\mu$ on
$\mathcal{S}$ such that the doubling condition
\begin{equation}\label{eq:doubling}
 \mu(B(x,2r)) \le C\mu(B(x,r))
\end{equation}
holds for all $x\in \mathcal{S}$ and $r>0$. As usual, the dilation of a ball $B(x,\lambda r)$ with $\lambda>0$ will be denoted by $\lambda B$.

If $C_\mu$ is the smallest constant for which \eqref{eq:doubling} holds,
then the number $D_\mu = \log_2 C_\mu$ is called the doubling order of
$\mu$. By iterating \eqref{eq:doubling}, we have
\begin{equation}\label{eq:doublingEXP}
\frac{\mu(B)}{\mu(\tilde{B})} \le
C^{2+\log_2\kappa}_{\mu}\left(\frac{r(B)}{r(\tilde{B})}\right)^{D_\mu} \;\mbox{for all
balls}\; \tilde{B} \subset B. 
\end{equation}
Note that $C^{2+\log_2\kappa}_{\mu}=(4\kappa)^{D_\mu}$. A particular case that we will use  is the following elementary inequality. Let $B$ be a ball and let $\lambda>1$. Then 
\begin{equation}\label{eq:doublingDIL}
 \mu(\lambda B) \le (2\lambda)^{D_\mu}   \mu(B)
\end{equation}

Throughout this paper, we will say that a constant $c=c(\kappa,\mu)>0$ is a \emph{structural constant} if it  depends only on the quasimetric constant $\kappa$ and the doubling constant $C_\mu$. The latter will often appear as a dependence on the doubling order $D_\mu$.

In a general space of homogeneous type, the balls $ B(x,r)$ are not necessarily open, but by a theorem of
Macias and Segovia \cite{MS}, there is a continuous quasimetric
$d'$ which is equivalent to $d$ (i.e., there are positive
constants $c_{1}$ and $c_{2}$ such that $c_{1}d'(x,y)\le d(x,y)
\le c_{2}d'(x,y)$ for all $x,y \in \mathcal{S}$) for which every ball is
open. We always assume that the quasimetric $d$ is continuous and
that balls are open. 

We will adopt the usual notation: if $\nu$ is a measure and $E$ is a measurable set, $\nu(E)$ denotes the $\nu$-measure of $E$. Also, if $f$ is a measurable function on $(\mathcal S,d,\mu)$ and $E$ is a measurable set, we will use the notation $f(E):=\int_E f(x)\ d\mu$. 
We also will denote the $\mu$-average of $f$ over a ball $B$ as 
\begin{equation*}
f_{B} = \avgint_B f d\mu = \frac{1}{\mu(B)} \int_B f d\mu. 
\end{equation*}
In several occasions we will need Lebesgue's differentiation theorem, so we will assume that it is valid for the spaces under study. This will follow if we assume that the family of continuous functions with compact support 
is a dense family in $L^1(\mu)$.

We recall that a weight $w$ (any non negative measurable function) satisfies the $A_p$ condition for $1<p<\infty$ if
\begin{equation*}
   [w]_{A_p}:=\sup_B\left(\avgint_B w\ d\mu\right)\left(\avgint_B w^{-\frac{1}{p-1}}\ d\mu\right)^{p-1},
\end{equation*}
where the supremum is taken over all the balls in $\mathcal{S}$. Since the $A_{p}$ classes are increasing with respect to $p$, we can define the $A_{\infty}$ class in the natural way by $A_{\infty}:=\bigcup_{p>1}A_p$. 
This class of weights can also be characterized by means of an appropriate constant. In fact, there are various different definitions of this constant, all of them equivalent in the sense that  they define the same class of weights. Perhaps the more classical and known definition is the following  due  to Hru\v{s}\v{c}ev
\cite{Hruscev} (see also \cite{GCRdF}):
\begin{equation*}
[w]^{exp}_{A_\infty}:=\sup_B \left(\avgint_{B} w\,d\mu\right) \exp \left(\avgint_{B} \log w^{-1}\,d\mu  \right).
\end{equation*}

However, in \cite{HP} the authors use a ``new'' $A_\infty$ constant (which was originally introduced  by Fujii in \cite{Fujii} and later by Wilson in \cite{Wilson:87}), which seems to be better suited. Let $M$ stands for the usual uncentered Hardy-Littlewood maximal operator:
\begin{equation*}
Mf(x) = \sup_{B\ni x } \avgint_{B} |f|\,d\mu.
\end{equation*}
The we can define
 \begin{equation*}
     [w]_{A_\infty}:= [w]^{W}_{A_\infty}:=\sup_B\frac{1}{w(B)}\int_B M(w\chi_B )\ d\mu.
 \end{equation*}
When the underlying space is $\mathbb{R}^d$, it is easy to see that $[w]_{A_\infty}\le c [w]^{exp}_{A_\infty}$ for some structural $c>0$.  In fact, it is shown in \cite{HP} that there are examples  showing that $[w]_{A_\infty}$ is much smaller than $[w]^{exp}_{A_\infty}$. The same line of ideas yields the inequality in this wider scenario.

\subsection{Main results}

\

We first present the statement of our main theorem in full generality, within the context of a space of homogeneous type $(\mathcal S,d,\mu)$. We have the following theorem:

\begin{theorem}[Sharp  weak Reverse H\"older Inequality]\label{thm:SharpRHI}

Let $w\in A_\infty$. Define the exponent $r(w)$ as 
\begin{equation*}
r(w)=1+\frac{1}{\tau_{\kappa\mu}[w]_{A_{\infty}}}:=
1+\frac{1}{ 6(32\kappa^2(4\kappa^2+\kappa)^2)^{D_\mu}[w]_{A_\infty}},
\end{equation*}
where $\kappa$ is the quasimetric constant and $D_\mu$ is the doubling order of the measure from \eqref{eq:doublingEXP}.

Then,
\begin{equation*}
  \left(\avgint_B w^{r(w)}\ d\mu\right)^{1/r(w)}\leq 2(4\kappa)^{D_\mu}\avgint_{2\kappa B} w\ d\mu,
\end{equation*}
where $B$ is any ball in $\mathcal S$.
\end{theorem}
 Note that on the right hand side, we have a dilation of the ball $B$, and that is the reason of calling this estimate a \emph{weak} inequality. We also remark here that the fact that on the right hand side we obtain a structural constant is crucial for the applications. We refer the reader to the work of Kinnunen \cite{Kinnunen} for an $A_1$ version of a reverse H\"older inequality.
 
As we already mentioned, from this theorem we will derive two results. The first is a ``precise open property'' for $A_p$ weights. We explicitly compute an admissible value of $\varepsilon>0$ such that any $A_p$ weight $w$ belongs to $A_{p-\varepsilon}$.

\begin{theorem} [The Precise Open property] \label{thm:Apselfimprove} Let $1<p<\infty$ and let $w\in
A_p$. Recall that for a weight $w$ we defined in Theorem \ref{thm:SharpRHI} the quantity $r(w)=1+\frac{1}{\tau_{\kappa\mu}[w]_{A_{\infty}}}$. Then  $w\in A_{p-\varepsilon} $ where 
\begin{equation*}
\varepsilon =\frac{p-1}{r(\sigma)' }= \frac{p-1}{ 1+\tau_{\kappa\mu}[\sigma]_{A_{\infty}} } 
\end{equation*}
where as usual $\sigma=w^{1-p'}$ and $p'$ is the dual exponent of $p$: $\frac{1}{p}+\frac{1}{p'}=1$. Furthermore,
\begin{equation}\label{eq:Ap-e}
 [w]_{A_{p-\varepsilon}} \le  2^{p-1}(4\kappa)^{pD_\mu}[w]_{A_p}
\end{equation} 
\end{theorem}

Finally, as an application of this last result, we present a short proof of the following mixed bound for the maximal function.

\begin{theorem}\label{thm:CorsharpBuckley}   Let $M$ be the Hardy-Littlewood maximal function and let  $1<p<\infty$ as above, and $\sigma=w^{1-p'}$. Then there is a structural constant $A_{\mu\kappa}$ such that 
\begin{equation}\label{eq:HLLpAinfty}
  \|M\|_{L^p(w)} \leq A_{\mu\kappa} \left( \frac{1}{p-1}  [w]_{A_p}[\sigma]_{A_\infty}\right)^{1/p},
\end{equation}
\end{theorem}

Of curse this theorem improves Buckley's theorem:
 $$\|M\|_{L^p(w)} \leq c_{\mu\kappa} p' [w]_{A_p}^{\frac{1}{p-1}}. $$

This paper is organized as follows. In Section \ref{sec:euclidean} we present the proof of the RHI in the simplest case: $\mathbb{R}^d$ with standard cubes and Lebesgue measure. Then, in Section \ref{sec:general} we present the  proof a weak version of the RHI for spaces of homogeneous type and then we derive the applications. We choose to present the two cases separately since in the first one the main ideas appear in a very clean way. Some of those arguments cannot be extended directly to the general case, so the proofs presented in Section \ref{sec:general} require an additional effort and we are able to obtain only a weak version of the RHI. Nevertheless, it will be shown that this weak version is good enough for the proof of the open property for $A_p$ weights and for the improvement of Buckley's theorem with sharp constants. 

\section{The proof for the classical setting}\label{sec:euclidean} \label{simplecase}
In this section we present the proof of the sharp RHI for $\mathbb{R}^d$ with the euclidean metric, Lebesgue measure and $A_p$ classes defined over cubes. 

The main advantage here is that we can use maximal functions adapted to those cubes. Since we will be working with dyadic children of an arbitrary cube $Q_0$, the appropriate definition of $A_\infty$ constant is the following.

\begin{definition}\label{def:Ainfty-cubes}
For a weight defined on $\mathbb{R}^n$, we define the $A_\infty$ constant as 
\begin{equation*}
     [w]_{A_\infty}:= \sup_Q\frac{1}{w(Q)}\int_Q M(w\chi_Q )\ dx.
 \end{equation*}
where the supremum is taken over \textbf{all} cubes with edges parallel to the coordinate axes. As usual, when this supremum is finite, we say that the weight $w$ belongs to the $A_\infty$ class.
\end{definition}

We start with the following lemma. It it interesting on its own, since it can be viewed as a self-improving property of the maximal function when restricted to $A_{\infty}$ weights.

\begin{lemma}\label{lem:RHIMaximal}
Let $w$ be  any $A_\infty$ weight and let  $Q_0$ be a cube. Then  for any 
$0<\varepsilon \le\frac{1}{2^{d+1}[w]_{A_\infty}}$,
we have that
\begin{equation}\label{eq:RHIMaximal}
 \avgint_{Q_0} (Mw)^{1+\varepsilon}\ dx\le 2[w]_{A_\infty}\left(\avgint_{Q_0} w\ dx\right)^{1+\varepsilon},
\end{equation} 
where $M$ denotes the \textbf{dyadic} maximal function associated to the cube $Q_0$.
\end{lemma}

\begin{proof}
We can assume, since all calculations will be performed on $Q_0$, that the weight $w$ is supported on that cube, that is, $w=w\chi_{Q_0}$. Define $\Omega_\lambda:=Q_0\cap\{Mw>\lambda \} $. 
We start with the following identity:

\begin{eqnarray*}
\int_{Q_0} (Mw)^{1+\varepsilon}\ dx  & = &  \int_0^\infty \varepsilon \lambda^{\varepsilon-1}Mw(\Omega_\lambda )\ d\lambda \\
& =  & \int_0^{ w_{{Q_0}}} \varepsilon \lambda^{\varepsilon-1}\int_{Q_0}Mw\ d\lambda + \int_{ w_{{Q_0}}}^\infty \varepsilon \lambda^{\varepsilon-1}Mw(\Omega_\lambda )\ d\lambda\\
\end{eqnarray*}
Now, for $\lambda\ge w_{Q_0}$, there is a family of maximal nonoverlapping dyadic cubes $\{Q_j\}_j$ for which 
\begin{equation*}
\Omega_\lambda=\bigcup_j Q_j \quad \text{ and }  \quad \avgint_{Q_j}w\ dx >\lambda.
\end{equation*}
Therefore, by using this decomposition and the definition of the  $A_\infty$ constant, we can write
\begin{equation}\label{eq:sum0}
\int_{Q_0} (Mw)^{1+\varepsilon}\ dx
 \le  w_{{Q_0}}^\varepsilon [w]_{A_\infty}w(Q_0) +\int_{w_{Q_0}}^\infty \varepsilon \lambda^{\varepsilon-1} \sum_j \int_{Q_j} Mw\ dx d\lambda.
\end{equation}
By maximality of the cubes in $\{Q_j\}_j$, it follows that the dyadic maximal function $M$ can be localized:
\begin{equation*}
 Mw(x)=M(w\chi_{Q_j})(x),
\end{equation*}
for any $x\in Q_j$, for all $j\in\mathbb{N}$. Now, if we denote by $\widetilde{Q}$ to the dyadic parent of a given cube $Q$, we have that
\begin{eqnarray*}
 \int _{Q_j}Mw\ dx =\int _{Q_j}M(w\chi_{Q_j})& \le &
[w]_{A_\infty}w(Q_j) \le [w]_{A_\infty}w(\widetilde{Q_j}) \\
& = & [w]_{A_\infty}w_{\widetilde{Q_j}} |\widetilde{Q_j}|\\
& \le & [w]_{A_\infty}\lambda 2^d|Q_j|
\end{eqnarray*}
Therefore, 
\begin{equation*}
\sum_j \int _{Q_j}Mw\ dx \le
\sum_j [w]_{A_\infty}\lambda 2^d|Q_j|\le [w]_{A_\infty}\lambda 2^d|\Omega_\lambda|,
\end{equation*}
and then \eqref{eq:sum0} becomes
\begin{equation*}
 \int_{Q_0} (Mw)^{1+\varepsilon}\ dx
 \le  w_{{Q_0}}^\varepsilon [w]_{A_\infty}w(Q_0) +\varepsilon[w]_{A_\infty}2^d\int_{w_{Q_0}}^\infty \lambda^{\varepsilon}|\Omega_\lambda|d\lambda.
\end{equation*}
Averaging over $Q_0$, we obtain that
\begin{equation*}
\avgint_{Q_0}(Mw)^{1+\varepsilon}\ dx \le   w_{Q}^{1+\varepsilon} [w]_{A_\infty} +
\frac{\varepsilon 2^d[w]_{A_\infty}}{1+\varepsilon}\avgint_{Q_0} (Mw)^{1+\varepsilon}\ dx.
\end{equation*}
To conclude with the proof, we can obtain the desired inequality for any $0<\varepsilon\le \frac{1}{2^{d+1}[w]_{A_\infty}}$ by absorbing the last term into the left.
\end{proof}
We now have the following theorem. 

We remark that in this standard case, we can recover (and improve) on the known sharp RHI, with no dilations involved.

\begin{theorem}[Sharp Reverse H\"older Inequality]\label{thm:SharpRHI-euclidean}
Let $w\in A_\infty$ and let $Q_0$ be a cube. Then
\begin{equation*}
\avgint_{Q_0} w^{1+\varepsilon}\ dx \le 2\left(\avgint_{Q_0} w \ dx\right)^{1+\varepsilon},
\end{equation*}
for any $\varepsilon>0$  such that $0<\varepsilon \le\frac{1}{2^{d+1}[w]_{A_\infty}-1 }$.  
\end{theorem}

Before we proceed with the proof, we remark here that the parameter $\frac{1}{2^{d+1}[w]_{A_\infty}-1 }$ is better than the one obtained in \cite{HP}.

\begin{proof}
 We assume again that $w=w\chi_{Q_0}$. We clearly have that
 \begin{equation*}
  \int_{Q_0}w^{1+\varepsilon}\ dx\leq  \int_{Q_0}(M w)^{\varepsilon}\ wdx.
 \end{equation*}
Now we argue in a similar way as in the previous lemma to obtain that
\begin{eqnarray*}
\int_{Q_0} (Mw)^{\varepsilon}\ wdx & =  & \int_0^\infty \varepsilon \lambda^{\varepsilon-1}w(\Omega_\lambda )\ d\lambda \\
&=& \int_0^{w_{Q_0}} \varepsilon \lambda^{\varepsilon-1}w(Q_0)\ d\lambda
+ \int_{w_{Q_0}}^\infty \varepsilon \lambda^{\varepsilon-1}w(\Omega_\lambda)\ d\lambda\\
&\le & w_{Q_0}^\varepsilon w(Q_0) + \int_{w_{Q_0}}^\infty \varepsilon \lambda^{\varepsilon-1} \sum_j w(Q_j)\ d\lambda,
\end{eqnarray*}
where the cubes $\{Q_j\}_j$ are from the decomposition of $\Omega_\lambda$ above. Therefore, 
\begin{eqnarray*}
\int_{Q_0} (Mw)^{\varepsilon}\ wdx  &\le & w_{Q_0}^\varepsilon w(Q_0) + \varepsilon 2^d\int_{w_{Q_0}}^\infty \lambda^{\varepsilon} \sum_j |Q_j|\ d\lambda\\
 &\le &w_{Q_0}^\varepsilon w(Q_0) + \varepsilon 2^d\int_{w_{Q_0}}^\infty \lambda^{\varepsilon} |\Omega_\lambda|\ d\lambda\\
 &\le &w_{Q_0}^\varepsilon w(Q_0) +  \frac{\varepsilon 2^d}{1+\varepsilon}\int_{Q_0} (Mw)^{1+\varepsilon}\ dx.
\end{eqnarray*}
Averaging over $Q_0$ we obtain
\begin{equation*}
\avgint_{Q_0}w^{1+\varepsilon}\ dx \le  w_{Q_0}^{1+\varepsilon}+ \frac{\varepsilon 2^d}{1+\varepsilon}\avgint_{Q_0} (Mw)^{1+\varepsilon}\ dx.
\end{equation*}

Now we use Lemma \ref{lem:RHIMaximal} to conclude with the proof:
\begin{eqnarray*}
\avgint_{Q_0} w^{1+\varepsilon}\ dx & \le &  w_{Q_0}^{1+\varepsilon}+ \frac{\varepsilon 2^d}{1+\varepsilon}
\avgint_{Q_0} (Mw)^{1+\varepsilon}\ dx\\
&\le & w_{Q_0}^{1+\varepsilon}+ 
\frac{\varepsilon 2^{d+1}[w]_{A_\infty}}{1+\varepsilon}\left(\avgint_{Q_0} w\ dx\right)^{1+\varepsilon}\\
&\le & 2\left(\avgint_{Q_0} w\ dx\right)^{1+\varepsilon}
\end{eqnarray*}
since, by hypothesis, $\frac{\varepsilon 2^d[w]_{A_\infty}}{1+\varepsilon}\le \frac{1}{2}$.
\end{proof}

\section{Proofs for the general case of spaces of homogeneous type}\label{sec:general}

Before we proceed with the proofs of our results, we need to introduce a local version of a Calder\'on-Zygmund lemma valid for spaces of homogeneous type from \cite{MacManus-Perez-98}. We need some notation first.

\begin{definition}
Let $B_0$ be a ball and let $\delta>0$ be fixed. We use the notation $\widehat{B_0} = (1+\delta)\kappa B_0$, where $\kappa$ is the quasimetric constant of $d$.  We also define the following family.
\begin{equation}\label{eq:B0}
{\mathcal B} = {\mathcal B}_{B_0,\delta} = \{B: x_B \in B_0
\ \mbox{ and }\ r(B)\le \delta  r(B_0)\}. 
\end{equation}
\end{definition}

Given an integrable function $f$ on $\widehat{B_0}$, the maximal function of $f$ associated to ${\mathcal B}$ is defined by
\begin{equation}\label{eq:maximallocal}
M_{ {\mathcal B}}f(x)= \sup_{B: x\in B\in {\mathcal B}} \avgint_B |f| d\mu
\end{equation}
if $x$ belongs to an element of the basis ${\mathcal B}$, and $M_{ {\mathcal B}}f(x)= 0$ otherwise.  We remark here that with this definition, the maximal operator depends on the reference ball $B_0$ and, in addition, on the parameter $\delta$, which can be any positive number.

Therefore, if $\lambda>0$ and we define
\begin{equation}\label{eq:omegalambda}
\Omega_\lambda = \{x \in \mathcal S:  M_{\mathcal B}f(x) >
\lambda\},
\end{equation} 
then 
\begin{equation}\label{eq:OmegaLambdaSUBwidehatB}
 \Omega_\lambda \subset \widehat{B_0}.
\end{equation}

An important property of the family $\mathcal B$ is expressed in the following lemma. The proof follows easily from \eqref{eq:doublingEXP}.
\begin{lemma}
Let $B_0$ be  any ball and let $\mathcal B$ be defined as in \eqref{eq:B0}. Let $f$ be a locally integrable function on $\widehat{B_0}$. Then,
\begin{enumerate}
 \item Any ball $B\in\mathcal{B}$ is contained in $\widehat{B_0}$.
\item 
If $\lambda < f_B$, then 
\begin{equation*}
r(B)\le 2\kappa^2(1+\delta)\left(\frac{f_{\widehat{B_0}}}{\lambda}\right)^{1/D_\mu} r(B_0)
\end{equation*}
It follows then that, for any $N>0$,
\begin{equation*}
 r(B)\le \frac{\delta}{N}r(B_0)
\end{equation*} 
whenever  $f_B >\lambda$ and $\lambda\ge \left(\frac{2\kappa^2(1+\delta)N}{\delta}\right)^{D_\mu}f_{\widehat{B_0}}$.
\end{enumerate}
\end{lemma}

We now present the local Calder\'on--Zygmund covering lemma in this context which is from \cite{MacManus-Perez-98}.
 
\begin{lemma}[Calder\'on--Zygmund decomposition]
\label{lem:localstoppingtime} Let $\delta>0$ and let $f$ be a nonnegative and integrable function on $\widehat{B_0}=(1+\delta)\kappa B_0$. For $N>0$ and  
$\lambda\ge \left(\frac{2\kappa^2(1+\delta)N}{\delta}\right)^{D_\mu}f_{\widehat{B_0}}$, 
define the set $\Omega_{\lambda}$ as in \eqref{eq:omegalambda}
If $\Omega_{\lambda}$ is not empty, then there exists a countable family
$\{B_i\}$ of pairwise disjoint balls such that
\begin{itemize}
\item[i)] $\displaystyle \cup_{i} B_{i}\subset \Omega_{\lambda} \subset
\cup_{i} B_{i}^{*}$,  where $B^{*}=(4\kappa^2+\kappa)B$.
\item[ii)]  $r(B_{i}) \le \frac{\delta}{N}r(B_0)$ for all $i$,
\item[iii)] For all $i$, 
\begin{equation*}
\lambda < \avgint_{B_i} f d\mu.
\end{equation*}
\item[iv)] 
If $\eta B_i \in {\mathcal B}$ and $\eta\ge 2$, then $f_{\eta B_i}\le \lambda$.
\end{itemize}
\end{lemma}

The next lemma contains a localization argument for the maximal function, which is  a key ingredient in the proof. For the dyadic case in $\mathbb{R}^d$, this was a direct consequence of the maximality of the cubes in the Calder\'on-Zygmund decomposition. In the general setting of spaces of homogeneous type we have the following substitute. We borrow the idea from \cite[Lemma 4.4]{MacManus-Perez-98}. 

\begin{lemma}\label{lem:localizedMaximal}
Consider, for a fixed $\lambda$ as in the previous lemma, the Calder\'on-Zygmund decomposition of the set $\Omega_\lambda$ with $N\ge 2\kappa$. Define $L=(8\kappa^2)^{D_\mu}$. Then, for any ball $B_i$ and any $x\in B_i^*\cap \Omega_{L\lambda}$, we have that
\begin{equation}
 M_\mathcal{B}f(x)\le M_\mathcal{B}(f\chi_{B_i^{**}})(x),
\end{equation} 
where $B^{**}=(B^*)^*=(4\kappa^2+\kappa)^2B$.
\end{lemma}

\begin{proof}
Let $x\in B_i^*\cap\Omega_{L\lambda}$. Then there exists a ball $B\in\mathcal{B}$ containing $x$ such that 
\begin{equation*}
L\lambda < \avgint_B |f|\ d\mu.
\end{equation*}
Now we claim that $r(B)\le r(B_i^*)$. Suppose not, then $r(B)>r(B_i^*)$ and therefore the ball $B$ is contained in $\tilde{B}:=B(x_i;2\kappa r(B))$. By the doubling property of $\mu$ from \eqref{eq:doublingEXP}, we have that 
\begin{equation*}
\avgint_B |f|\ d\mu\le \frac{\mu(\tilde B)}{\mu(B)}\avgint_{\tilde B} |f|\ d\mu\le L\avgint_{\tilde B} |f|\ d\mu.
\end{equation*}
The ball $\tilde B$ clearly belongs to $\mathcal{B}$, since 
$$
r(\tilde{B})=2\kappa r(B)\le \frac{2\kappa\delta}{N}r(B_0)\le \delta r(B_0). 
$$
In addition, under the hypothesis that $r(B)>r(B_i^*)$, we have that $\tilde{B}=2\kappa \frac{r(B)}{r(B_i)}B_i=\eta B_i$ with $\eta>2$. Therefore, property iv) of Lemma \ref{lem:localstoppingtime} implies that $|f|_{\tilde{B}}\le\lambda$. This implies that $|f|_B\le L\lambda$, which is a contradiction. Then the claim is true and $r(B)\le r(B_i^*)$. It is clear that in this case the ball $B$ is contained in $B_i^{**}$ and then 
$$
\avgint_B |f|\ d\mu\le\avgint_B |f|\chi_{B_i^{**}}\ d\mu\le M_\mathcal{B}(f\chi_{B_i^{**}})(x).
$$

\end{proof}

Now we present the proof of the Reverse H\"older inequality and its applications. We start with a preliminary lemma, which is the generalization of Lemma \ref{lem:RHIMaximal}.

\begin{lemma}\label{lem:generalRHMaximal}
Let $w$ be  any $A_\infty$ weight. Then there is a structural constant $\tau_{\mu\kappa}$ such that, for any 
$0<\varepsilon \le\frac{1}{\tau_{\mu\kappa\delta}[w]_{A_\infty}}$,
we have that
\begin{equation}\label{eq:generalRHMaximal}
 \avgint_{\widehat{B}} (M_\mathcal{B}w)^{1+\varepsilon}\ d\mu\le 3[w]_{A_\infty}\left(\avgint_{\widehat B} w\ d\mu\right)^{1+\varepsilon}.
\end{equation} 
The constant $\tau_{\mu\kappa\delta}$ can be taken as $\tau_{\mu\kappa\delta}= 6(16\kappa^2(4\kappa^2+\kappa)^2(1+\frac{1}{\delta}))^{D_\mu}$, where $\kappa$ is the quasimetric constant and $D_\mu$ is the doubling order of the measure from \eqref{eq:doublingEXP}.
\end{lemma}

\begin{proof}
As in the proof of Lemma \ref{lem:RHIMaximal}, we can assume that the weight is localized, in this case on the ball $\widehat B$, namely  $w=w\chi_{\widehat{B}}$, since by definition of the local maximal function, the values of $w$ outside $\widehat{B}$ are ignored (recall \eqref{eq:OmegaLambdaSUBwidehatB}). We write
\begin{equation}\label{eq:omegalambda2}
 \Omega_\lambda =\{x\in \widehat B: M_\mathcal{B}w(x)>\lambda\}.
\end{equation}
For $N\ge (4\kappa^2+\kappa)^2$ and 
$\Gamma=\left(\frac{2\kappa^2(1+\delta)N}{\delta}\right)^{D_\mu}$, we write
\begin{eqnarray*}
\int_{\widehat{B}} (M_\mathcal{B}w)^{1+\varepsilon}\ d\mu & =  & \int_0^\infty \varepsilon \lambda^{\varepsilon-1}M_\mathcal{B}w(\Omega_\lambda )\ d\lambda \\
& =  & \int_0^{\Gamma w_{\widehat{B}}} \varepsilon \lambda^{\varepsilon-1}M_\mathcal{B}w(\Omega_\lambda )\ d\lambda + \int_{\Gamma w_{\widehat{B}}}^\infty \varepsilon \lambda^{\varepsilon-1}M_\mathcal{B}w(\Omega_\lambda)\ d\lambda\\
&=& \Gamma^\varepsilon w_{\widehat{B}}^\varepsilon\int_{\widehat{B}} M_\mathcal{B}w\ d\mu
+ \int_{\Gamma w_{\widehat{B}}}^\infty \varepsilon \lambda^{\varepsilon-1}M_\mathcal{B}w(\Omega_\lambda)\ d\lambda\\
\end{eqnarray*}
Now we use the Calder\'on-Zygmund decomposition from Lemma \ref{lem:localstoppingtime} with  the choice for $N$ above.
For the first term, since we are only considering the values of $w$ on $\widehat B$, we can use the definition of the $A_\infty$ constant. Then we obtain 
\begin{equation}\label{eq:sum}
\int_{\widehat{B}} (M_\mathcal{B}w)^{1+\varepsilon}\ d\mu  \le  \Gamma^\varepsilon w_{\widehat{B}}^\varepsilon [w]_{A_\infty}w({\widehat{B}}) +\int_{\Gamma w_{\widehat{B}}}^\infty \varepsilon \lambda^{\varepsilon-1} \sum_i \int_{B_i^*} M_\mathcal{B}w\ d\mu d\lambda,
\end{equation}
where the family $\{B_i\}_i$ have the properties listed in that lemma. Now we focus on a fixed $B_i^*$ and compute the integral of the maximal function as follows. Consider $L$ as in Lemma \ref{lem:localizedMaximal}, $L=(8\kappa^2)^{D_\mu} $ and the partition of $B_i^*=B_1\cup B_2 \cup B_3$ where 
\begin{equation*}\label{eq:B1B2B3}
B_1=B_i^*\cap \Omega_{L\lambda}, \qquad B_2=B_i^*\cap \Omega_{\lambda}\setminus \Omega_{L\lambda}, \qquad B_3= B_i^*\setminus \Omega_{\lambda}.
\end{equation*}
Then, using Lemma \ref{lem:localizedMaximal}, we obtain 
\begin{eqnarray*}
\int_{B_i^*}M_\mathcal{B}w\ d\mu & = & \int_{B_1}M_\mathcal{B}w\ d\mu +\int_{B_2}M_\mathcal{B}w\ d\mu +\int_{B_3}M_\mathcal{B}w\ d\mu\\
 & \le & \int_{B_i^{**}}M_\mathcal{B}(w\chi_{B_i^{**}})\ d\mu + L\lambda \mu(B_i^*)+\lambda \mu(B_i^*),\\
\end{eqnarray*}
where we use in the last term that the inclusion $B_i^*\subset {\widehat{B}}$ implies that $M_\mathcal{B}w(x)\le\lambda$ for any $x\in B_i^*\setminus \Omega_\lambda$.

Now define, to abbreviate, $\theta=4\kappa^2+\kappa$. Then $B^*=\theta B$ and by the doubling property \eqref{eq:doublingDIL}, we have that $\mu(B_i^*)\le (2\theta)^{D_\mu}\mu(B_i)$. Now, again by  definition of $[w]_{A_\infty}$, we have
\begin{eqnarray*}
\int_{B_i^*}M_\mathcal{B}w\ d\mu &  \le &  [w]_{A_\infty}w(B_i^{**}) + 2L\lambda (2\theta)^{D_\mu}\mu(B_i)\\
  & \le & \left(  [w]_{A_\infty}w_{B_i^{**}}(2\theta)^{2D_\mu} +  2L\lambda (2\theta)^{D_\mu}\right)\mu(B_i)\\
  & \le & 3L(2\theta)^{2D_\mu}[w]_{A_\infty}\lambda \mu(B_i)
\end{eqnarray*}
since, by the choice of $N$, the average of the weight over $B_i^{**}$ is smaller than $\lambda$. Now we can continue with the sum from \eqref{eq:sum}:

\begin{eqnarray*}
\int_{w_{\widehat{B}}}^\infty \varepsilon \lambda^{\varepsilon-1} \sum_i \int_{B_i^*} M_\mathcal{B}w\ d\mu d\lambda & \le & 3L(2\theta)^{2D_\mu}[w]_{A_\infty}\int_{\Gamma w_{\widehat{B}}}^\infty \varepsilon \lambda^{\varepsilon} \sum_i \mu(B_i)\ d\lambda\\
& \le & 3L(2\theta)^{2D_\mu}[w]_{A_\infty}\int_{\Gamma w_{\widehat{B}}}^\infty \varepsilon \lambda^{\varepsilon} \mu(\Omega_\lambda)\ d\lambda\\
& \le & \frac{3\varepsilon L(2\theta)^{2D_\mu}[w]_{A_\infty}}{1+\varepsilon}\int_{\widehat{B}} (M_\mathcal{B}w)^{1+\varepsilon}\ d\mu\\
\end{eqnarray*}
Finally, collecting all estimates and taking the average over ${\widehat{B}}$, we obtain 
\begin{equation*}
\avgint_{\widehat{B}} (M_\mathcal{B}w)^{1+\varepsilon}\ d\mu  \le  \Gamma^\varepsilon w_{\widehat{B}}^{1+\varepsilon} [w]_{A_\infty} +
\frac{3\varepsilon L(2\theta)^{2D_\mu}[w]_{A_\infty}}{1+\varepsilon}\avgint_{\widehat{B}} (M_\mathcal{B}w)^{1+\varepsilon}\ d\mu.
\end{equation*}
By the hypothesis on $\varepsilon$, we have that 
\begin{equation*}
 \frac{3\varepsilon L(2\theta)^{2D_\mu}[w]_{A_\infty}}{1+\varepsilon}\le \frac{1}{2},
\end{equation*}
and therefore the last term can be absorbed by the left hand side. In addition, one can verify (with some tedious computations) that $\Gamma^\varepsilon\le \frac{3}{2}$ for $\varepsilon\le\frac{1}{\tau_{\mu\kappa\delta}}$. We obtain 
\begin{equation*}
\avgint_{\widehat{B}} (M_\mathcal{B}w)^{1+\varepsilon}\ d\mu  \le 3[w]_{A_\infty}w_{\widehat{B}}^{1+\varepsilon},
\end{equation*}
and the proof is complete.
\end{proof}

Now we are ready to present the new proof of the weak Reverse H\"older Inequality with sharp exponent.

\begin{proof}[Proof of Theorem \ref{thm:SharpRHI}]
Let $w$ be an $A_\infty$ weight
 and let $B$ be a fixed ball. We remark here that  the idea is to use Lemma \ref{lem:generalRHMaximal} where we made the assumption on the localization of the weight, namely $w=w\chi_{\widehat B}$.
Note that in that lemma, the maximal operator depends on $\delta$, and any positive $\delta$ will work. But due to the blow-up of $\varepsilon$ on the endpoint $\delta=0$, we will choose $\delta$ away from 0, namely $\delta=1$. Therefore, the hypothesis on $\varepsilon$ that we will use are that $0<\varepsilon \le\frac{1}{\tau_{\mu\kappa}[w]_{A_\infty}}$, with $\tau_{\mu\kappa}= 6(32\kappa^2(4\kappa^2+\kappa)^2)^{D_\mu}$. With this assumption on $\delta$, the inequality we need to prove is the following:
\begin{equation}\label{eq:weakRHI}
\left(\avgint_B w^{1+\varepsilon}\ d\mu\right)^\frac{1}{1+\varepsilon}\leq 2(4\kappa)^{D_\mu}\avgint_{\widehat B} w\ d\mu, 
\end{equation}
for the any $\varepsilon>0$ as above.

As in Section \ref{sec:euclidean} above, we can bound the weight by the maximal function. Then,
 \begin{equation}\label{eq:wBoundedbyMaximal}
  \int_Bw^{1+\varepsilon}\ d\mu  \leq  \int_{\widehat{B}}(M_\mathcal B w)^{\varepsilon}\ wd\mu.
 \end{equation}

For the term on the right hand side, we proceed in a similar way as in the previous lemma, considering a Calder\'on-Zygmund decomposition of the level set 
\begin{equation*}
 \Omega_\lambda =\{x\in \widehat{B}: M_\mathcal{B}w(x)>\lambda\}.
\end{equation*}
with  $N\ge \kappa(4\kappa+1)$. Then,
\begin{eqnarray*}
\int_{\widehat{B}} (M_\mathcal{B}w)^{\varepsilon}\ wd\mu & =  & \int_0^\infty \varepsilon \lambda^{\varepsilon-1}w(\Omega_\lambda )\ d\lambda \\
&=& \int_0^{\Gamma w_{\widehat{B}}} \varepsilon \lambda^{\varepsilon-1}w(\widehat{B})\ d\lambda
+ \int_{\Gamma w_{\widehat{B}}}^\infty \varepsilon \lambda^{\varepsilon-1}w(\Omega_\lambda)\ d\lambda\\
&\le& \Gamma^\varepsilon w_{\widehat{B}}^\varepsilon w(\widehat{B}) + \int_{\Gamma w_{\widehat{B}}}^\infty \varepsilon \lambda^{\varepsilon-1} \sum_i w(B_i^*)\ d\lambda\\
\end{eqnarray*}
Where $\Gamma=\left(4\kappa^2N\right)^{D_\mu}$ and we use, as before, that $B_i^*=\theta B_i = (4\kappa^2+\kappa)B_i$ and thus
\begin{equation*}
w(B_i^*)=w_{B_i^*}\mu(B_i^*)\le \lambda (2\theta)^{D_\mu}\mu(B_i).
\end{equation*}
Therefore,
\begin{eqnarray*}
\int_{\widehat{B}} (M_\mathcal{B}w)^{\varepsilon}\ wd\mu  &\le& \Gamma^\varepsilon  w_{\widehat{B}}^\varepsilon w(\widehat{B}) + \varepsilon (2\theta)^{D_\mu}\int_{\Gamma w_{\widehat{B}}}^\infty \lambda^{\varepsilon} \sum_i w(B_i)\ d\lambda\\
 &\le&\Gamma^\varepsilon  w_{\widehat{B}}^\varepsilon w(\widehat{B}) + \varepsilon (2\theta)^{D_\mu}\int_{\Gamma w_{\widehat{B}}}^\infty \lambda^{\varepsilon} \mu(\Omega_\lambda)\ d\lambda\\
 &\le& \Gamma^\varepsilon w_{\widehat{B}}^\varepsilon w(\widehat{B}) + 
 \frac{\varepsilon (2\theta)^{D_\mu}}{1+\varepsilon}\int_{\widehat{B}} (M_\mathcal{B}w)^{1+\varepsilon}\ d\mu.
\end{eqnarray*}
Then, averaging in \eqref{eq:wBoundedbyMaximal} over $\widehat B$, we obtain
\begin{equation*}
\frac{\mu(B)}{\mu(\widehat{B})}\avgint_B w^{1+\varepsilon}\ d\mu \le \Gamma^\varepsilon w_{\widehat B}^{1+\varepsilon}+ \frac{\varepsilon (2\theta)^{D_\mu}}{1+\varepsilon}
\avgint_{\widehat{B}} (M_\mathcal{B}w)^{1+\varepsilon}\ d\mu.
\end{equation*}
Now we note that by hypothesis we have that $\varepsilon$ is in the range  allowed in Lemma \ref{lem:generalRHMaximal}. We use again the doubling property \eqref{eq:doublingDIL} and then we obtain the desired estimate:
\begin{eqnarray*}
\avgint_B w^{1+\varepsilon}\ d\mu & \le &  (4\kappa)^{D_\mu} \left(\Gamma^\varepsilon w_{\widehat B}^{1+\varepsilon}+ \frac{\varepsilon (2\theta)^{D_\mu}}{1+\varepsilon}\avgint_{\widehat B} (M_\mathcal{B}w)^{1+\varepsilon}\ d\mu\right)\\
&\le & (4\kappa)^{D_\mu} \left(\Gamma^\varepsilon + \frac{\varepsilon (2\theta)^{D_\mu}3[w]_{A_\infty}}{1+\varepsilon}\right)\left(\avgint_{\widehat B} w\ d\mu\right)^{1+\varepsilon}\\
& \le & 2(4\kappa)^{D_\mu} \left(\avgint_{\widehat B} w\ d\mu\right)^{1+\varepsilon}
\end{eqnarray*}
To check last inequality, it is easy to verify that with this choice of $\varepsilon$, we also have that $ \frac{\varepsilon (2\theta)^{D_\mu}3[w]_{A_\infty}}{(1+\varepsilon)}\le \frac{1}{2}$ and, as before, $\Gamma^\varepsilon \le \frac{3}{2}$.
This completes the proof of the weak version of the RHI stated in \eqref{eq:weakRHI}.
\end{proof}

\subsection{Precise open property for    Muckenhoupt  classes}

\begin{proof}[Proof of Theorem \ref{thm:Apselfimprove}] Let $w \in A_p$ and denote, as usual, the dual weight $w^{1-p'}=\sigma$. We choose $\varepsilon =\frac{p-1}{r(\sigma)' }$, which is the same as $r(\sigma)=\frac{p-1}{p-\varepsilon-1}$ (observe that $\varepsilon>0$ and $p-\varepsilon>1$). We can easily compute the following

\begin{eqnarray*}
\left(\avgint_B w^{1-(p-\varepsilon)'}\ d\mu\right)^{p-\varepsilon-1}
& = & \left(\avgint_B w^{(1-p')r(\sigma)}\ d\mu\right)^\frac{p-1}{r(\sigma)}\\
& \le & \left(2(4\kappa)^{D_\mu}\avgint_{\widehat B} \sigma\ d\mu\right)^{p-1}
\end{eqnarray*}
by the sharp weak RHI. Now, for the $A_{p-\varepsilon}$ constant of $w$, we proceed as follows. Let $B$ be any ball. Then, by the doubling property \eqref{eq:doublingDIL} of the measure, we have that
\begin{equation*}
\avgint_B w\ d\mu \left(\avgint_B w^{1-(p-\varepsilon)'}\right)^{p-\varepsilon-1}
\le  2^{p-1}(4\kappa)^{pD_\mu }
\avgint_{\widehat B} w\ d\mu\left(\avgint_{\widehat B} \sigma\right)^{p-1}
\end{equation*} 
Taking the supremum over all balls, we obtain 
\begin{equation*}
 [w]_{A_{p-\varepsilon}} \le  2^{p-1}(4\kappa)^{pD_\mu}[w]_{A_p},
\end{equation*} 
and therefore the proof is complete.
\end{proof}

\subsection{Sharp Buckley's theorem with mixed constants}

\begin{proof}[Proof of Theorem \ref{thm:CorsharpBuckley}]
We start by pointing out that we have the following analogue of the standard case of $\mathbb{R}^d$ (with cubes) for the known weak norm estimate for the maximal function:
\begin{equation}\label{eq:Maximal-weak}
\|M\|_{L^{q,\infty}(w)} \le (2\theta)^{D_\mu }[w]_{A_q}^{\frac1q}     \qquad 1<q<\infty,
\end{equation} 
where  $\theta=4\kappa^2+\kappa$ and $\kappa$ is the quasimetric constant. 
Consider, for any nonnegative measurable function $f$ and $\lambda>0$, the level set $\Omega_\lambda=\{x\in \mathcal{S}: Mf(x)>\lambda\}$. By a Vitali type covering lemma (\cite{SW}, Lemma 3.3) we can obtain a countable family of balls $\{B_j\}_j$ such that 
\begin{equation*}
\frac{1}{\mu(B_j)}\int_{B_j}f\ d\mu >\lambda \qquad \mbox{ and }\quad 
\Omega_\lambda \subset \bigcup_j B_j^*
\end{equation*}
where, as before, $B^*=\theta B$. Therefore
\begin{eqnarray*}
 \lambda^qw(\Omega_\lambda) & \le & \lambda^q\sum_j w(B_j^*)\\
& \le & \sum_j w(B_j^*)\left(\frac{1}{\mu(B_j)}\int_{B_j}fw^\frac{1}{q}w^{-\frac{1}{q}}\ d\mu\right)^q\\
& \le & (2\theta)^{D_\mu q}\sum_j \frac{w(B_j^*)}{\mu(B^*_j)}\left(\frac{1}{\mu(B^*_j)}\int_{B^*_j}\sigma\ d\mu\right)^{q-1}
\left(\int_{B_j}f^qw\ d\mu\right)^q\\
&\le &   (2\theta)^{D_\mu q}[w]_{A_q}\|f\|^q_{L^q(w)}
\end{eqnarray*}
and then \eqref{eq:Maximal-weak} follows. We will also use that, for $f_t:=f\chi_{f>t}$, the following inclusion holds:
\begin{equation*}
 \{x\in  \mathcal S: Mf(x)>2t \} \subset \{x\in\mathcal S: Mf_t(x)>t \}
\end{equation*}

Now, we write the integral from the $L^p$ norm and compute
\begin{eqnarray*}
\|Mf\|_{L^p(w)}^p & = & p\int_{0}^{\infty} t^{p} w \{y\in \mathcal S:Mf(y) > t\} \frac{dt}{t} \\
& = & p2^p \int_{0}^{\infty} t^{p} w \{y\in \mathcal S:Mf(y) > 2t\} \frac{dt}{t}\\
& \leq & p 2^p \int_{0}^{\infty}  t^{p} w \{y\in \mathcal S:Mf_t(y) > t\}
  \frac{dt}{t}\\
& \leq  & p 2^p (2\theta)^{D_\mu (p-\varepsilon)} [w]_{A_{p-\varepsilon}} \int_{0}^{\infty}  t^{p} \int_{\mathcal S} 
  \frac{  f_t^{p-\varepsilon} }  {t^{p-\varepsilon}}  w\ d\mu  \frac{dt}{t}\\
& \leq & p2^{2p-1}(4\kappa)^{pD_\mu} (2\theta)^{D_\mu (p-\varepsilon)}  [w]_{A_{p}}\int_{\mathcal S}   \int_{ 0 }^{ f(y) } t^{\varepsilon}\frac{dt}{t}  f^{p-\varepsilon}  w\ d\mu \\
& = &   p2^{2p-1} (4\kappa)^{pD_\mu} (2\theta)^{D_\mu (p-\varepsilon)} \frac{[w]_{A_{p}}}{\varepsilon} \int_{\mathcal S}    f^{p}  w\ d\mu
\end{eqnarray*}
By the precise open property we can take $\varepsilon=\frac{p-1}{r(\sigma)' }= \frac{p-1}{ 1+\tau_{\kappa\mu}[\sigma]_{A_{\infty}} } $. With this choice, we finally obtain that 
\begin{equation*}
\|M\|_{L^p(w)}^p \le  
\frac{p2^{2p-1} (4\kappa)^{pD_\mu}(2\theta)^{D_\mu (p-\varepsilon)} }{p-1}( 1+\tau_{\kappa\mu}[\sigma]_{A_{\infty}})[w]_{A_{p}} \int_{\mathcal S}    f^{p}  w\ d\mu,
\end{equation*}
and this yields \eqref{eq:HLLpAinfty} and therefore the proof is complete.
\end{proof}


\end{document}